\documentclass[final,3p,times,10pt]{elsarticle}
\usepackage[centertags]{amsmath}
\usepackage{amsfonts,amssymb,amsthm,epstopdf,newlfont,graphicx}
\usepackage{fixltx2e,booktabs,bm,enumitem,setspace,empheq,cancel}
\usepackage{algorithm}
\usepackage[short]{optidef}
\usepackage{algpseudocode}
\usepackage{caption,subcaption}
\usepackage{threeparttable,multirow}
\usepackage{amsmath,accents}
\usepackage[T1]{fontenc}
\usepackage[american]{babel}
\usepackage{natbib}
\usepackage{appendix}
\usepackage[colorlinks=true]{hyperref}

\newcommand\hcancel[2][black]{\setbox0=\hbox{$#2$}%
\rlap{\raisebox{.25\ht0}{\textcolor{#1}{\rule{0.7\wd0}{0.75pt}}}}#2} 
\newcommand\hcancelt[2][black]{\setbox0=\hbox{$#2$}%
\rlap{\raisebox{.25\ht0}{\textcolor{#1}{\hspace{0.3mm}\rule{0.7\wd0}{0.75pt}}}}#2} 
\biboptions{comma,square,sort&compress}
\newtheorem{thm}{Theorem}[section]

\newtheorem{cor}{Corollary}[section]
\newtheorem{rem}{Remark}[section]
\theoremstyle{definition}

\numberwithin{algorithm}{section}
\numberwithin{equation}{section}
\renewcommand{\theequation}{\thesection.\arabic{equation}}
\def\simgt{\,\hbox{\lower0.6ex\hbox{$>$}\llap{\raise0.3ex\hbox{$\sim$}}}\,}
\def\simlt{\,\hbox{\lower0.6ex\hbox{$<$}\llap{\raise0.3ex\hbox{$\sim$}}}\,}
\def\simgteq{\,\hbox{\lower0.6ex\hbox{$\ge$}\llap{\raise0.6ex\hbox{$\sim$}}}\,}
\def\simlteq{\,\hbox{\lower0.6ex\hbox{$\le$}\llap{\raise0.6ex\hbox{$\sim$}}}\,}
\def\applteq{\,\hbox{\lower0.6ex\hbox{$\le$}\llap{\raise0.8ex\hbox{$\approx$}}}\,}
\def\applt{\,\hbox{\lower0.6ex\hbox{$<$}\llap{\raise0.5ex\hbox{$\approx$}}}\,}
\DeclareMathAlphabet\mathbfcal{OMS}{cmsy}{b}{n}

\DeclareMathOperator{\csch}{csch}

\DeclareMathOperator{\Span}{span}
\DeclareMathOperator{\Mvec}{vec}
\makeatletter
\def\user@resume{resume}
\def\user@intermezzo{intermezzo}
\newcounter{previousequation}
\newcounter{lastsubequation}
\newcounter{savedparentequation}
\renewenvironment{subequations}[1][]{%
      \def\user@decides{#1}%
      \setcounter{previousequation}{\value{equation}}%
      \ifx\user@decides\user@resume 
           \setcounter{equation}{\value{savedparentequation}}%
      \else  
      \ifx\user@decides\user@intermezzo
           \refstepcounter{equation}%
      \else
           \setcounter{lastsubequation}{0}%
           \refstepcounter{equation}%
      \fi\fi
      \protected@edef\theHparentequation{%
          \@ifundefined {theHequation}\theequation \theHequation}%
      \protected@edef\theparentequation{\theequation}%
      \setcounter{parentequation}{\value{equation}}%
      \ifx\user@decides\user@resume 
           \setcounter{equation}{\value{lastsubequation}}%
         \else
           \setcounter{equation}{0}%
      \fi
      \def\theequation  {\theparentequation  \alph{equation}}%
      \def\theHequation {\theHparentequation \alph{equation}}%
      \ignorespaces
}{%
  \ifx\user@decides\user@resume
       \setcounter{lastsubequation}{\value{equation}}%
       \setcounter{equation}{\value{previousequation}}%
  \else
  \ifx\user@decides\user@intermezzo
       \setcounter{equation}{\value{parentequation}}%
  \else
       \setcounter{lastsubequation}{\value{equation}}%
       \setcounter{savedparentequation}{\value{parentequation}}%
       \setcounter{equation}{\value{parentequation}}%
  \fi\fi
  \ignorespacesafterend
}
\makeatother
\newcommand{\C}[1]{\mathcal{#1}}

\newcommand{\F}[1]{\mathbf{#1}}

\newcommand{\bsC}[1]{\boldsymbol{\C{#1}}}

\newcommand{\MB}[1]{\mathbb{#1}}
\newcommand{\MBS}{\MB{S}}
\newcommand{\MBR}{\mathbb{R}}
\newcommand{\MBRP}{\MBR^+}

\newcommand{\MBRzer}{\MBR_0}
\newcommand{\MBRzerP}{\MBRzer^+}
\newcommand{\MBZ}{\mathbb{Z}}
\newcommand{\MBZP}{\MBZ^+}
\newcommand{\MBZzer}{\MBZ_0}
\newcommand{\MBZzerP}{\MBZzer^+}
\newcommand{\MBZe}{\MBZ_e}
\newcommand{\MBZeP}{\MBZe^+}

\newcommand{\MBT}{\mathbb{T}}
\newcommand{\MBJ}{\mathbb{J}}
\newcommand{\MBF}{\mathfrak{F}}

\newcommand{\MBC}{\mathfrak{C}}

\newcommand{\Fthe}{\F{\Theta}}
\newcommand{\canczer}[1]{#1_{\hcancel{0}}}
\newcommand{\cancbra}[1]{\hcancel{[}#1\hcancelt{]}}
\newcommand{\sumd}{\sideset{}{'}}
\newcommand{\bmc}{\bm{c}}

\newcommand{\bT}{\bar T}

\newcommand{\hu}{\hat{u}}
\newcommand{\bmx}{\bm{x}}
\newcommand{\bbmx}{\bar \bmx}

\newcommand{\bmt}{\bm{t}}

\newcommand{\tbmx}{\tilde{\bmx}}
\newcommand{\tbmu}{\tilde{\bmu}}

\newcommand{\tbmxs}{\tbmx^*}
\newcommand{\tbmus}{\tbmu^*}

\newcommand{\bmy}{\bm{y}}

\newcommand{\bmh}{\bm{h}}
\newcommand{\bmxs}{\bm{x}^*}
\newcommand{\bmu}{\bm{u}}
\newcommand{\bmus}{\bm{u}^*}
\newcommand{\bmf}{\bm{f}}
\newcommand{\bmF}{\bm{F}}
\newcommand{\dbmx}{\dot{\bm{x}}}
\newcommand{\bmzer}{\bm{\mathit{0}}}
\newcommand{\bmone}{\bm{\mathit{1}}}

\newcommand{\hbmx}{\hat\bmx}

\newcommand{\FOmega}{\F{\Omega}}
\newcommand{\foralla}{\,\forall_{\mkern-6mu a}\,}
\newcommand{\foralle}{\,\forall_{\mkern-6mu e}\,}
\newcommand{\foralls}{\,\forall_{\mkern-6mu s}\,}

\usepackage{mathtools}
\mathtoolsset{showonlyrefs}
\makeatletter
\def\BState{\State\hskip-\ALG@thistlm}
\makeatother
\MakeRobust{\eqref}
\errorcontextlines\maxdimen

\makeatletter
    \newcommand*{\algrule}[1][\algorithmicindent]{\makebox[#1][l]{\hspace*{.5em}\thealgruleextra\vrule height \thealgruleheight depth \thealgruledepth}}%
\newcommand*{\thealgruleextra}{}
\newcommand*{\thealgruleheight}{.75\baselineskip}
\newcommand*{\thealgruledepth}{.25\baselineskip}

\newcount\ALG@printindent@tempcnta
\def\ALG@printindent{%
    \ifnum \theALG@nested>0
        \ifx\ALG@text\ALG@x@notext
        \else
            \unskip
            \addvspace{-1pt}
            \ALG@printindent@tempcnta=1
            \loop
                \algrule[\csname ALG@ind@\the\ALG@printindent@tempcnta\endcsname]%
                \advance \ALG@printindent@tempcnta 1
            \ifnum \ALG@printindent@tempcnta<\numexpr\theALG@nested+1\relax
            \repeat
        \fi
    \fi
    }%
\usepackage{etoolbox}
\patchcmd{\ALG@doentity}{\noindent\hskip\ALG@tlm}{\ALG@printindent}{}{\errmessage{failed to patch}}
\makeatother

\newbox\statebox
\newcommand{\myState}[1]{%
    \setbox\statebox=\vbox{#1}%
    \edef\thealgruleheight{\dimexpr \the\ht\statebox+1pt\relax}%
    \edef\thealgruledepth{\dimexpr \the\dp\statebox+1pt\relax}%
    \ifdim\thealgruleheight<.75\baselineskip
        \def\thealgruleheight{\dimexpr .75\baselineskip+1pt\relax}%
    \fi
    \ifdim\thealgruledepth<.25\baselineskip
        \def\thealgruledepth{\dimexpr .25\baselineskip+1pt\relax}%
    \fi
    \State #1%
    \def\thealgruleheight{\dimexpr .75\baselineskip+1pt\relax}%
    \def\thealgruledepth{\dimexpr .25\baselineskip+1pt\relax}%
}
\begin{document}
\begin{frontmatter}
\title{Numerical Solution of Nonlinear Periodic Optimal Control Problems Using a Fourier Integral Pseudospectral Method}
\author[XMUM,Assiut]{Kareem T. Elgindy\corref{cor1}}
\ead{kareem.elgindy@(xmu.edu.my;gmail.com)}
\address[XMUM]{Mathematics Department, School of Mathematics and Physics, Xiamen University Malaysia, Sepang 43900, Malaysia}
\address[Assiut]{Mathematics Department, Faculty of Science, Assiut University, Assiut 71516, Egypt}
\cortext[cor1]{Corresponding author}

\begin{abstract}
This paper presents a Fourier integral pseudospectral (FIPS) method for a general class of nonlinear, periodic optimal control (OC) problems with equality and/or inequality constraints and sufficiently smooth solutions. In this scheme, the integral form of the problem is collocated at an equispaced set of nodes, and all necessary integrals are approximated using highly accurate Fourier integration matrices (FIMs). The proposed method leads to a nonlinear programming problem (NLP) with algebraic constraints, which is solved using a direct numerical optimization method. Sharp convergence and error estimates are derived for Fourier series, interpolants, and quadratures used for smooth and continuous periodic functions. Two nonlinear examples are considered to show the accuracy and efficiency of the numerical scheme. 
\end{abstract}
\end{frontmatter}
\section{Introduction}
\label{Int}
The improvements in performances and efficiency induced by periodic operations and controls have been increasingly valuable in the current world. The optimization of periodic processes and phenomena leads naturally to periodic optimal control (OC) problems, which has been the subject of several studies and applications such as computing periodic orbits for space systems, flight enhancements and optimal mission planning, automobile test-driving, reduction of railway noise and vibration effects in mechanical systems, reduction of energy demand of continuous distillation processes, controlling periodic adsorbers and reactors, the design of walking robots, production planning, sustainable harvesting of ecological, biological, and energy resources, determining fast switching periodic protocols for COVID-19; cf. 
\cite{Mulholland1975148,dorato1979periodic,Bernstein1980kz,DeKlerk1982253,
pittelkau1993optimal,Sachs1993572,arcara1997active,Kasturi1998499,
bausa2001reducing,van2003optimization,mombaur2005open,xiao2006optimal,
williams2006direct,sager2008fast,Sachs2011,Guo2011305,Sachs2012,
Sachs201215,belyakov2014constant,zuyev2017isoperimetric,caruso2018semi,
gao2019hypersonic,fu2020harvesting,lippiello2022estimating,
jin2022railway,zhang2022optimal,Elgindy2022b}. For autonomous problems, optimal periodic controls that improve the steady-state solution exist and can be derived using the $\pi$-test ; cf. \cite{Guardabassi1971kz,Bernstein1980kz}.

Several methods were presented in the literature to solve periodic OC problems. The maximum principle was perhaps the standard approach to solve periodic OC problems in the 1970s and 1980s; cf. \cite{partain1971line,kowler1972optimal,matsubara1973periodic,noldus1974comments,
getz1979optimal,bittanti1986optimal,colonius1986global,rink1988optimal,
nitka1989approximate}. More recent applications of the maximum principle for solving periodic OC problems can be found in \cite{Wang2003429,Bayen2015750,Yan2016847,rouot2017optimal,bettiol2019sub,korolev2020solution,
bayen2020optimal,ali2021maximizing}. Other methods include a describing function approach \cite{Bertele1972kz}, Fourier series expansions \cite{Evans1979622,epperlein2020frequency,moretti2021quadratic}, Hermite collocation \cite{Dickmanns1984137}, asymptotic expansions \cite{Dmitriev1987219,Chuang1988678,Evans1987343,Kurina200823}, shooting methods \cite{Sachs1993572,Maurer1998185,vanNoorden20034115,Houska20062693,
sager2008fast}, finite-difference schemes \cite{Maurer1998185,Williams20062144,zhang2022optimal}, Fourier pseudospectral (PS) method \cite{Elnagar2004707}, Hermite-Simpson, Hermite-Legendre-Gauss-Lobatto, Chebyshev-Legendre, and Jacobi-PS methods \cite{Williams20062144}, value iteration method \cite{azzato2008applying}, Markov chain and symplectic methods \cite{Wang2009160}, multiharmonic finite element (FEM) approximations \cite{Langer20161267,Wolfmayr20201050,
liang2021robust}, preconditioning techniques \cite{axelsson2019note}, adaptive dynamic programming \cite{pang2020adaptive}, modified Q-Learning method \cite{dao2021q}, etc. For periodic problems with smooth solutions, however, the series expansion and function classes that have proven to be the most successful by far are Fourier series expansions and Fourier interpolant functions for many reasons including, but not limited to: (i) both classes converge exponentially fast with increasing expansion terms faster than any polynomial rates \cite{gottlieb1977numerical,Fornberg1978373,Adams1986Slab,
Fornberg1996practical,Canuto1988}, (ii) the periodicity conditions are satisfied automatically when the solution is represented by any of the two classes without additional constraints, (iii) the coefficients of differentiated/integrated Fourier expansions/spectral interpolants are easily obtained \cite{Fornberg1996practical}, (iv) the conversion between Fourier coefficients and the function values at the collocation/interpolation nodes is very fast through the FFT algorithm, (v) Fourier basis functions can represent a wider range of frequencies than Chebyshev and Legendre basis functions often used in common direct PS methods for solving OC problems, which means that they can capture more detail in the solution, leading to greater accuracy, (vi) Legendre basis functions are only defined on a finite interval, which makes them less appealing for problems with periodic boundary conditions, as they can introduce discontinuities at the boundaries, (vii) for large expansion terms, the derivative of Fourier interpolants can be efficiently computed using the FFT in just $O(N \log_2 N)$ operations, where $N$ is the number of terms of Fourier approximation \cite{kopriva2009}, (viii) the Fourier differentiation matrix for the first derivative is skew-symmetric, which ensures stability of discretization for evolutionary, linear PDEs with time-varying coefficients \cite{Hairer2016751}, (ix) although the Fourier integration matrix (FIM) is a dense matrix, its first row is a zero row and the remaining rows form ``a row-wise element-twins matrix'' in the sense that each element in each row has exactly one twin element in the same row allowing for faster computations \cite{elgindy2019high}, and (x) the FIM constructed using equi-spaced nodes on any finite interval of length $T$ can be normalized and generated efficiently using a $T$-invariant constant ``basic/principle/generating/natural'' FIM, which can be constructed and stored offline and invoked later when running the solver codes \cite{elgindy2023optimal}. These reasons and more furnish little incentives to seek alternative bases than Fourier basis functions for periodic OC problems. 

In this work, we propose to solve generally nonlinear, periodic OC problems exhibiting sufficiently smooth solutions using an integral PS (IPS) method\footnote{An IPS method is a robust variant of a PS method that is aka as a PS integration method.} based on Fourier basis functions. Although PS methods are known for being easier, more flexible to implement, and computationally more efficient than the alternative Galerkin and tau approximations, especially in the presence of nonlinearities, variable coefficients, and boundary conditions; cf. \cite{Zang1982485,Peyret1996kz,Weideman2000465}, IPS methods have the additional advantage of reformulating the dynamical system equations into their integral form first as a prerequisite before the collocation step starts; thus, avoids the usual ill-conditioning associated with numerical differentiation processes. One may perform the integral reformulation by either a direct integration of the dynamical system equations in the presence of constant coefficients, or by representing the highest-order derivative of the solution involved in the problem by a nodal finite series in terms of its grid point values and then solve for those grid point values before successively integrating back to obtain the sought solution grid point values in a stable manner. Our proposed method follows the former approach and enjoys all of the aforesaid merits; in addition, it (i) can be easily programmed, (ii) can straightforwardly handle inequality constraints on the state- and control-variables, (iii) dispenses the need to solve for the adjoint variables associated with the maximum principle, which often require precise initial guesses, and runs without sophisticated transversality conditions, and (iv) directly discretizes the periodic OC problem on any time interval $[0, T]$, for $T > 0$, without the necessity to transform the domain first into the interval $[0, 2 \pi]$ as applied earlier in \cite{Elnagar2004707}. To the best of our knowledge, the present paper introduces the first direct Fourier IPS (FIPS) method to date for solving periodic OC problems exhibiting smooth solutions in which the OC problem is discretized into a nonlinear programming problem (NLP) in conjuction with FIMs by means of FIPS methods and the solution is sought in the physical space. For periodic OC problems exhibiting bang-bang solutions, the reader may consult our recent works in \cite{Elgindy2022b,elgindy2023optimal} in which the solutions can be recovered within excellent accuracies using an FIPS method and an adaptive h-IPS method\footnote{An adaptive h-IPS method has the ability to recover the discontinuous/non-smooth solutions with high accuracy via the decomposition of the solution interval into smaller mesh intervals or elements ($h$-refinement), and approximating the restricted solution on each element with a finite, nodal expansion series in terms of the solution grid point values by means of interpolation \cite{Elgindy2022b}.} composed through a predictor-corrector algorithm. 
For periodic fractional OC problems in which the system dynamics are described by fractional derivatives that preserve the periodicity of periodic functions, the reader may consult our recent works in \cite{elgindy2023fouriera,elgindy2023fourierb}. The reader may consult further \cite{ElgindyHareth2022a} for a list of the advantages of direct IPS methods compared with other common methods for solving OC problems, such as indirect and parameterization methods.

The organization of the paper is as follows. In the next section, we give some preliminary notations that we adopt throughout the paper. Section \ref{sec:PS1} contains a statement of the periodic OC problem. Section \ref{sec:FPSIIM} presents the necessary tools we shall use to discretize the OC problem. Section \ref{sec:CRFSPF1} presents a rigorous study and sharp estimates on the convergence rates and errors of Fourier series, interpolants, and quadratures for smooth $T$-periodic functions. Section \ref{sec:FPIMIRF1} outlines the method of approach. Numerical results on two test problems are contained in Section \ref{sec:CRAC1} followed by some concluding remarks in Section \ref{sec:Conc} and a future work in Section \ref{sec:FW1}.

\section{Preliminary Notations}
\label{sec:PN}
\noindent\textbf{Logical Symbols.} $\forall, \foralla, \foralle$, and $\foralls$ stand for the phrases ``for all,'' ``for any,'' ``for each,'' and ``for some,'' respectively.\\[0.5em]
\textbf{Set and List Notations.} The symbols $\MBC$ and $\MBF$ denote the sets of all complex-valued functions and all real-valued functions; moreover, $\canczer{\MBZ}, \MBZP, \MBZzerP, \MBZeP$, and $\MBRzerP$ denote the sets of non-zero integers, positive integers, non-negative integers, positive even integers, and non-negative real numbers, respectively. The notations $i:j:k$ or $i(j)k$ indicate a list of numbers from $i$ to $k$ with increment $j$ between numbers, unless the increment equals one where we use the simplified notation $i:k$. For example, $0:0.5:2$ simply means the list of numbers $0, 0.5, 1, 1.5$, and $2$, while $0:2$ means $0, 1$, and $2$. The set of any numbers $y_1, y_2, \ldots, y_n$ is represented by $\{y_{1:n}\}$. We define $\MBJ_n = \{0:n-1\}$ and $\MBJ'_n = \MBJ_n\backslash\{0\}\,\foralla n \in \MBZP$; moreover, $\MB{K}_N = \{-N/2:N/2\}$ and $\MB{K}'_N = \MB{K}\backslash\{N/2\}\,\foralla N \in \MBZeP$. Also, $\MBS_n = \left\{t_{0:n-1}\right\}$ is the set of $n$ equally-spaced points such that $t_j = T j/n\, \forall j \in \MBJ_n$.\\[0.5em]
\textbf{Function Notations.} For convenience, we shall denote $g(t_{n})$ by $g_n \foralla g \in \MBC$, unless stated otherwise.\\[0.5em]
\textbf{Space Notations.} $\MBT_T$ is the space of $T$-periodic, univariate functions $\foralla T \in \MBRP$. $C^k(\FOmega)$ is the space of $k$ times continuously differentiable functions on ${\FOmega}\,\forall k \in \MBZzerP$.\\[0.5em]
\textbf{Vector Notations.} We shall use the shorthand notation ${\bmt_N}\;(\text{or }$ $t_{0:N-1}^t)$ to stand for the column vector $[t_{0}, t_{1}, \ldots, t_{N-1}]^t$. $g_{0:N-1}$ and $g^{(0:n)}$ denote the column vector $[g_0, g_1, \ldots, g_{N-1}]^t$ and the column vector of derivatives $[g, g', \ldots, g^{(n)}]^t\,\forall n \in \MBZzerP$ in respective order. In general, $\foralla h \in \MBF$ and row/column vector $\bmy$ whose $i$th-element is $y_i \in \MBR$, the notation $h(\bmy)$ stands for a vector of the same size and structure of $\bmy$ such that $h(y_i)$ is the $i$th element of $h(\bmy)$. Moreover, by $\bmh(\bmy)$ or $h_{1:m}\cancbra{\bmy}$ with a stroke through the square brackets, we mean $[h_1(\bmy), \ldots, h_m(\bmy)]^t\,\foralla m$-dimensional column vector function $\bmh$, with the realization that the definition of each array $h_i(\bmy)$ follows the former notation rule $\foralle i$. If $\bmy$ is a vector function, say $\bmy = \bmy(t)$, then we write $h(\bmy(\bmt_N))$ and $\bmh(\bmy(\bmt_N))$ to denote $[h(\bmy(t_0)), h(\bmy(t_1)), \ldots, h(\bmy(t_{N-1}))]^t$ and $[\bmh(\bmy(t_0)), \bmh(\bmy(t_1)), \ldots, \bmh(\bmy(t_{N-1}))]^t$ in respective order. One can naturally extend these notations into higher dimensions; for instance, $h(\bmy(\bmt_N), \bmt_N)$ and $\bmh(\bmy(\bmt_N), \bmt_N)$ simply means the column vector $[h(\bmy(t_0), t_0), h(\bmy(t_1), t_1), \ldots,$ $h(\bmy(t_{N-1}), t_{N-1})]^t$ and the $N \times m$ matrix $[\bmh(\bmy(t_0), t_0), \bmh(\bmy(t_1), t_1), \ldots, \bmh(\bmy(t_{N-1}), t_{N-1})]^t$, respectively, and so on.\\[0.5em] 
\textbf{Interval Notations.} The specific interval $[0, c]$ is denoted by $\FOmega_c\,\forall c > 0$. For example, $[0, t_{n}]$ is denoted by ${\FOmega_{t_{n}}}$; moreover, ${\FOmega_{t_{0:N-1}}}$ stands for the list of intervals ${\FOmega_{t_{0}}}, {\FOmega_{t_{1}}}, \ldots, {\FOmega_{t_{N-1}}}$.\\[0.5em] 
\textbf{Integral Notations.} By closely following the convention for writing definite integrals introduced in \citep{elgindy2019high}, we denote $\int_0^{{t_{l}}} {h(t)\,dt}$ by $\C{I}_{{t_{l}}}^{(t)}h \foralla$ integrable $h \in \MBT_T$. If the integrand function $h$ is to be evaluated at any other expression of $t$, say $u(t)$, we express $\int_0^{{t_{l}}} {h(u(t))\,dt}$ with a stroke through the square brackets as $\C{I}_{{t_{l}}}^{(t)}h\cancbra{u(t)}$. We adopt the notation $\C{I}_{{\bmt_{N}}}^{(t)}h$ to denote the $N$th-dimensional column vector $\left[ {\C{I}_{{t_{0}}}^{(t)}h,\C{I}_{{t_{1}}}^{(t)}h, \ldots ,\C{I}_{{t_{N - 1}}}^{(t)}h} \right]^t$. Furthermore, we write $\C{I}_{{\bmt_{N}}}^{(t)}\bmh$ to denote the $N \times m$ matrix $\left[ \C{I}_{{t_{0}}}^{(t)}\bmh,\C{I}_{{t_{1}}}^{(t)}\bmh, \ldots ,\right.$ $\left.\C{I}_{{t_{N - 1}}}^{(t)}\bmh \right]^t\,\foralla m$-dimensional vector function $\bmh$.\\[0.5em] 
\textbf{Matrix Notations.} $\F{O}_n, \F{1}_n$, and $\F{I}_n$ stand for the zero, all ones, and the identity matrices of size $n$. $\F{C}_{n,m}$ indicates that $\F{C}$ is a rectangular matrix of size $n \times m$; moreover, $\F{C}_n$ denotes a row vector whose elements are the $n$th-row elements of $\F{C}$, except when $\F{C}_n = \F{O}_n, \F{1}_n$, or $\F{I}_n$, where it denotes the size of the matrix. For convenience, a vector is represented in print by a bold italicized symbol while a two-dimensional matrix is represented by a bold symbol, except for a row vector whose elements form a certain row of a matrix where we represent it in bold symbol as stated earlier. For example, $\bmone_n$ and $\bmzer_n$ denote the $n$-dimensional all ones- and zeros- column vectors, while $\F{1}_n$ and $\F{O}_n$ denote the all ones- and zeros- matrices of size $n$, respectively. Finally, the notations $[.;.]$ and $\Mvec(.)$ denote the usual vertical concatenation and vectorization of a matrix, respectively.

\section{Problem Statement}
\label{sec:PS1}
Let $m, n, p \in \MBZP, T \in \MBRP, \C{U} = \{\bmu: \MBRzerP \to \MBR^m\text{ s.t. }\bmu\text{ is a }$ $T\text{-periodic, smooth vector function}\}$, and consider the dynamical system model
\begin{subequations}
\begin{equation}\label{eq:BRKM1}
\dbmx(t) =  \bmf\left(\bmx(t),\bmu(t),t\right),\quad \forall t \in {\FOmega_T},
\end{equation}
subject to the periodic boundary conditions
\begin{equation}\label{eq:PBC1}
\bmx(0) = \bmx(T),
\end{equation}
and the inequality path constraints
\begin{equation}\label{eq:PBC2}
\bmc(\bmx(t), \bmu(t), t) \le \bmzer_p,
\end{equation}
where $\bmx: \MBRzerP \to \MBR^n, \bmu \in \C{U}, \bmf = (f_i)_{1 \le i \le n}: \MBR^n \times \MBR^m \times \MBRzerP \to \MBR^n$, and $\bmc = (c_i)_{1 \le i \le p}: \MBR^n \times \MBR^m \times \MBRzerP \to \MBR^p$ with $f_i, c_i \in C^{k}(\MBRzerP)\,\foralls k \ge 1$. For a given time period $T$, we aim to find the optimal $T$-periodic waveforms $\bmxs: \MBRzerP \to \MBR^n$ and $\bmus \in \C{U}$, which satisfy Conditions \eqref{eq:BRKM1}-\eqref{eq:PBC2} and minimize the performance index functional
\begin{equation}\label{eq:OC1}
J(\bmu) = \frac{1}{T} \C{I}_{T}^{(t)} {g}\cancbra{\bmx(t), \bmu(t), t},
\end{equation}
\end{subequations}
where $g: \MBR^n \times \MBR^m \times \MBRzerP \to \MBR$ such that $g \in C^{k}(\MBRzerP)\,\foralls k \ge 1$. We refer to this problem by Problem $\C{P}$. Such a problem is a periodic, finite-horizon OC problem in Lagrange form where $\bmx$ and $\bmu$ are the state and control variables, respectively. We assume that both $g$ and $\bm{f}$ are $T$-periodic. 
If we integrate both sides of Eq. \eqref{eq:BRKM1} over the time interval $\FOmega_t \foralls t \in {\FOmega_T}\backslash\{0\}$, we transform the OC problem into its integral form where the same performance index $J$ is minimized subject to the integral equation
\begin{equation}\label{eq:IDS1}
\bmx(t) = \bmx(0) + \C{I}_t^{(\tau)} {\bmf\cancbra{\bmx(\tau), \bmu(\tau), \tau}},
\end{equation}
and Conditions \eqref{eq:PBC1} and \eqref{eq:PBC2}. We refer to this integral form of Problem $\C{P}$ by Problem $\C{IP}$. Although Problems $\C{P}$ and $\C{IP}$ are mathematically equivalent, they are not necessarily numerically equivalent in floating-point arithmetic. In particular, Problem $\C{IP}$ often admits better approximate solutions in practice due to the well-conditioning of numerical integration operators in general; cf. see \cite{greengard1991spectral,heath2002scientific,elbarbary2007p,Elgindy20171,Elgindy2019b,Elgindy2020distributed} and the Refs. therein. 

\section{FPS Interpolation and Integration Matrices}
\label{sec:FPSIIM}
Let $v^*$ be the complex conjugate of $v\, \foralla v \in \MBC, N \in \MBZeP, t_{j} \in \MBS_N\,\forall j \in \MBJ_N, f_j = f(t_{j})\,\foralla f \in \MBT_T$, and consider the $N/2$-degree, $T$-periodic Fourier interpolant, ${I_N}f$, such that $\left( {{I_N}f} \right)_{0:N - 1} = f_{0:N - 1}$ so that
\begin{equation}\label{eq:FI1nn1}
{I_N}f(t) = \sumd\sum\limits_{\left| k \right| \le N/2} {{\tilde f_k} {e^{i{\omega _k}t}}},
\end{equation}
where ${\omega _{\alpha}} = \displaystyle{\frac{{2\pi \alpha}}{T}}\,\forall \alpha \in \MBR, {\tilde f_k}$ is the discrete Fourier interpolation coefficient given by
\[{\tilde f_k} = \frac{1}{N}\sum\limits_{j = 0}^{N - 1} {{f_j}{e^{ - i \omega_k {t_{j}}}}},\quad \forall k \in \MB{K}_N,\]
and the primed sigma denotes a summation in which the last term is omitted. Notice that $\tilde f_{N/2} = \tilde f_{-N/2}$, and the DFT pair is defined by
\begin{subequations}
\begin{empheq}[left=\empheqbiglbrace]{alignat=2}
  {{\tilde f}_k} &= \frac{1}{N}\sum\limits_{j = 0}^{N - 1} {{f_j}{e^{ - i{\omega _k}{t_{j}}}}}  = \frac{1}{N}\sum\limits_{j = 0}^{N - 1} {{f_j}{e^{ - i{{\hat \omega }_{jk}}}}}, &&\quad k \in \MB{K}'_N,\label{DFP1}\\
  {f_j} &= \sumd\sum\limits_{\left| k \right| \le N/2} {{{\tilde f}_k}{e^{i{\omega _k}{t_{j}}}}}  = \sumd\sum\limits_{\left| k \right| \le N/2} {{{\tilde f}_k}{e^{i{{\hat \omega }_{jk}}}}}, &&\quad \forall j \in \MBJ_N,\label{DFP2}
\end{empheq}
\end{subequations}
where ${{\hat \omega }_k} = 2\pi k/N\;\forall k$. Substituting Eq. \eqref{DFP1} into Eq. \eqref{eq:FI1nn1}, and then swapping the order of the summations, express the interpolant in terms of the following function grid point values form
\begin{equation}\label{eq:eqLF1}
{I_N}f(t) = \sum\limits_{j = 0}^{N - 1} {{f_j}{\C{F}_j}(t)},
\end{equation}
where ${\C{F}_j}(t)$ is the $N/2$-degree, $T$-periodic trigonometric Lagrange interpolating polynomial given by
\[{\C{F}_j}(t) = \frac{1}{N}\sumd\sum\limits_{\left| k \right| \le N/2} {e^{i{\omega _k}(t - {t_{j}})}} = {\left[ {\frac{1}{N}\sin \left( {\frac{{\pi N}}{T}\left( {t - {t_{j}}} \right)} \right)\cot \left( {\frac{\pi }{T}\left( {t - {t_{j}}} \right)} \right)} \right]_{t \ne {t_{j}}}},\]
$\forall j \in \MBJ_N$. Notice that $\C{F}_j(t_{l}) = \delta_{j,l}\,\forall j,l \in \MBJ_N$, where $\delta_{j,l}$ is the kronecker delta function of variables $j$ and $l$. The Fourier PS interpolation operator ${I_N}f$ defined by \eqref{eq:eqLF1} is an orthogonal projection on the space $\Span\left\{ {{e^{i \omega_k t}}:k \in \MB{K}'_N} \right\}$ with the discrete $L^2$ inner product
\[{(u,v)_N} = \frac{T}{N}\sum\limits_{j = 0}^{N - 1} {{u_j}v_j^*},\quad \foralla u, v \in \MBC.\]
Integrating ${I_N}f$ over the interval $\FOmega_{t_{l}}$ yields
\begin{equation}\label{eq:AppRedFIM1}
\C{I}_{{t_{l}}}^{(t)}({I_N}f) = \sum\limits_{j = 0}^{N - 1} {{\theta _{l,j}}{f_j}},\quad \forall l \in \MBJ_N,
\end{equation}
where 
\begin{equation}\label{eq:RedFIM1}
{\theta _{l,j}} = \frac{1}{N}\left[ {{t_{l}} + \frac{{Ti}}{{2\pi }} \sumd\sum\limits_{\scriptstyle \left| k \right| \le N/2\atop
\scriptstyle k \ne 0} {\frac{1}{k}{e^{ - i{\omega _k}{t_{j}}}}\left( {1 - {e^{i{\omega _k}{t_{l}}}}} \right)}} \right],\quad \forall l,j \in \MBJ_N,
\end{equation}
are the entries of the first-order square FIM, $\Fthe$, of size $N$; cf. \cite{Elgindy2022b}. The definite integrals of ${I_N}f$ over the intervals $\FOmega_{y_{0:M-1}}\,\foralls M$-random set of points $\{ {{y_{0:M-1}}}\} \subset \FOmega_T\backslash\{0\}:{y_{l}} \notin {\MBS_N}\forall M \in {\MBZ^ + },l \in \MBJ_M$ are given by
\begin{equation}\label{eq:AppRedFIM2}
\C{I}_{{y_{l}}}^{(t)}({I_N}f) = \sum\limits_{j = 0}^{N - 1} {{\hat \theta _{l,j}}{f_j}},\quad \forall l \in \MBJ_M,
\end{equation}
where 
\begin{equation}\label{eq:RedMFIM1}
{\hat{\theta}_{l,j}} = \C{I}_{{y_{l}}}^{(t)}{\C{F}_j} = \frac{1}{N}\left[ {{y_{l}} + \frac{{Ti}}{{2\pi }}\sumd\sum\limits_{\scriptstyle\left| k \right| \le N/2\atop
\scriptstyle k \ne 0} {\frac{1}{k}{e^{ - i{\omega _k}{t_j}}}\left( {1 - {e^{i{\omega _k}{y_{l}}}}} \right)}} \right],
\end{equation}
$\forall l \in \MBJ_M, j \in \MBJ_N,$ are the elements formulas of the associated rectangular FIM, $\F{\hat{\Theta}} = \left(\hat \theta_{l,j}\right): l \in \MBJ_M, j \in \MBJ_N$. One can further write Formulas \eqref{eq:AppRedFIM1} and \eqref{eq:AppRedFIM2} in matrix notation as
\[{{\C{I}_{{{\bmt_N}}}^{(t)}(I_Nf)}} = \Fthe f_{0:N - 1}\quad \text{and}\quad {{\C{I}_{{{\bm{y}_M}}}^{(t)}(I_Nf)}} = \F{\hat{\Theta}} f_{0:N - 1},\]
respectively. In the special case when $y_{l} = T$, Formula \eqref{eq:RedMFIM1} reduces to ${\hat{\theta}_{l,j}} = T/N\, \forall j \in \MBJ_N$. For convenience, we denote ${\hat{\theta}_{l,j}}$ by $\theta_{N,j}$ in this particular case and define $\Fthe_N = \frac{T}{N} \bmone_N^t$ so that 
\begin{equation}\label{eq:PCase13Feb221}
{{\C{I}_{{{T}}}^{(t)}(I_Nf)}} = \Fthe_N f_{0:N - 1} = \frac{T}{N} \left(\bmone_N^t f_{0:N - 1}\right).
\end{equation}
Further information about the peculiar structure, characteristics, and algorithmic construction of FIMs can be found in \cite{elgindy2019high,Elgindy2022b}.

\section{Errors and Convergence Rate for Smooth, $T$-Periodic Functions} 
\label{sec:CRFSPF1}
In this section, we study the convergence rate for smooth $T$-periodic functions and the error analysis of their truncated Fourier series, interpolation operators, and integration operators. Let $\bm{\beta} = [-\beta, \beta]\,\forall \beta > 0$, 
\[{\F{C}_{T,\beta} } = \left\{ {x + iy:x \in {\FOmega_T},y \in \bm{\beta}} \right\},\quad \forall \beta  > 0,\]
and $L^p({\FOmega_T})$ be the Banach space of measurable functions $u$ defined on ${\FOmega_T}$ such that ${\left\| u \right\|_{{L^p}}} = {\left( {{\C{I}_{\FOmega_T}}{{\left| u \right|}^p}} \right)^{1/p}} < \infty$. Let also
\[\displaystyle{{H^s}({\FOmega_T}) = \left\{ {u \in {L_{loc}}({\FOmega_T}),\;{D^\alpha }u \in {L^2}({\FOmega_T}),\left| \alpha  \right| \le s\;} \right\}},\quad \forall s \in \MBZzerP,\]
be the inner product space with the inner product $\displaystyle{{(u,v)_s}} =$\\ $\displaystyle{\sum\nolimits_{\left| \alpha  \right| \le s} {\C{I}_{{\FOmega_T}} ^{(x)}\left( {{D^\alpha }u\,{D^\alpha }v} \right)}}$, where ${{L_{loc}}({\FOmega_T} )}$ is the space of locally integrable functions on ${\FOmega_T}$ and ${{D^\alpha }u}$ denotes any derivative of $u$ with multi-index $\alpha$. Define
\[\C{H}_T^s = \left\{ {u \in {H^s}({\FOmega_T}),\;{u^{(s)}} \in {BV},\;{u^{(0:s - 1)}}(0) = {u^{(0:s - 1)}}(T)} \right\},\]
where $\displaystyle{{BV} = \left\{ {u \in {L^1}({\FOmega_T}):{{\left\| u \right\|}_{BV}} < \infty } \right\}}$ with the norm $\displaystyle{{{\left\| u \right\|}_{BV}} =}$\\ $\displaystyle{\sup \left\{ {\C{I}_T^{(x)}(u\phi '),\;\phi  \in \C{D}({\FOmega_T}),\;{{\left\| \phi  \right\|}_{{L^\infty }}} \le 1} \right\}}$ such that 
\[\C{D}({\FOmega_T}) = \left\{ {u \in {C^\infty }({\FOmega_T}):{\text{supp}}(u){\text{ is a compact subsect of }}{\FOmega_T}} \right\}.\]
Define also
\[\C{A}_{T,\beta} = \{u \in \C{H}_T^{\infty}: u\text{ is analytic in some open set containing }\F{C}_{T,\beta}\},\]
with the norm ${\left\| u \right\|_{{\C{A}_{T,\beta}}}} = {\left\| u \right\|_{{L^\infty }({{\F{C}}_{T,\beta}})}}$. For convenience of writing, we shall denote ${\left\|  \cdot  \right\|_{{L^2}({\FOmega_T})}}$ and $e^{i \omega_k x}$ by $\left\|  \cdot  \right\|$ and $\phi_k(x)\,\forall k$, respectively, and call a function $u \in \C{A}_{T,\beta}$ \textit{``a $\beta$-analytic function''} if $u$ is analytic on ${\F{C}_{T,\infty}}$ and $\displaystyle{{\lim _{\beta  \to \infty }}\frac{{{{\left\| u \right\|}_{{\C{A}_{T,\beta} }}}}}{{{e^{{\omega _\beta }}}}} = 0}$.

\begin{thm}[Decay of Fourier Series Coefficients for analytic, $T$-periodic functions]\label{thm:0}
Suppose that $f \in {{\C{A}_{T,\beta}}} \foralls \beta > 0$, is approximated by the $N/2$-degree, $T$-periodic truncated Fourier series 
\begin{equation}\label{eq:FTS1}
{\Pi _N}f(x) = \sum\limits_{\left| k \right| \le N/2} {{{\hat f}_k}{\phi_k(x)}},\quad \forall N \in \MBZeP,
\end{equation}
where $\hat f_{-N/2:N/2}$ is the Fourier series coefficients vector of $f$, then 
\begin{equation}\label{eq:fhatkMAR312021new1}
\left| {{{\hat f}_k}} \right| = O\left({e^{ - {\omega _{\left|k\right| \beta }}}}\right),\quad \text{ as }\left|k\right| \to \infty.
\end{equation}
\end{thm}
\begin{proof}
Notice first that the set of complex exponentials $\displaystyle{\left\{ {{\phi _k}} \right\}_{k =  - N/2}^{N/2}}$ is orthogonal on ${\FOmega_T}$ with respect to the weight function $w(x) = 1\,\forall x \in {\FOmega_T}$ such that $\left( {{\phi _n},{\phi _m}} \right) = \C{I}_T^{(x)}\left( {{\phi _n}\,\phi _m^*} \right) = T{\delta _{n,m}}$, where $\delta _{n,m}$ is the Kronecker delta function defined by 
\[{\delta _{n,m}} = \left\{ \begin{array}{l}
1,\quad n = m,\\
0,\quad n \ne m.
\end{array} \right.\]
Therefore, $\left( {{\phi _n},{\phi _n}} \right) = \C{I}_T^{(x)}\left( {{\phi _n}\,\phi _n^*} \right) = \C{I}_T^{(x)}\left( {{{\left| {{\phi _n}} \right|}^2}} \right) = {\left\| {{\phi _n}} \right\|^2} = T$. Fourier coefficients, $\hat f_k$, of $f$ can thus be determined via the
orthogonal projection $(f,\phi_k)$, which produces
\begin{align}
{\hat f_k} &= \frac{1}{T}(f,{\phi _k}) = \frac{1}{T}\C{I}_T^{(x)}\left( {f{\mkern 1mu} {\phi _{ - k}}} \right)\nonumber\\
&= \left\{ \begin{array}{l}
\frac{1}{T}\C{I}_T^{(x)}{\left(f\cancbra{x - i\beta}{\mkern 1mu} {\phi _{ - k}}\cancbra{x - i\beta}\right)}\;\forall k \ge 0,\\
\frac{1}{T}\C{I}_T^{(x)}\left({f\cancbra{x + i\beta}{\mkern 1mu} {\phi _{ - k}}\cancbra{x + i\beta}} \right)\;\forall k < 0
\end{array} \right.\nonumber\\
&= \left\{ \begin{array}{l}
\frac{{{e^{ - {\omega _{k\beta }}}}}}{T}\C{I}_T^{(x)}\left({f\cancbra{x - i\beta}{\mkern 1mu} {\phi _{ - k}}} \right)\;\forall k \ge 0,\\
\frac{{{e^{ - {\omega _{ - k\beta }}}}}}{T}\C{I}_T^{(x)}\left({f\cancbra{x + i\beta}{\mkern 1mu} {\phi _{ - k}}} \right)\;\forall k < 0.
\end{array} \right.\label{eq:19APr2021_1}
\end{align}
Therefore,
\begin{equation}\label{eq:fhatkMAR312021}
\left| {{{\hat f}_k}} \right| \le {\left\| f \right\|_{{\C{A}_{T,\beta} }}}{e^{ - {\omega _{\left|k\right| \beta }}}}\quad \forall k \in \MB{K}_N,
\end{equation}
from which the Asymptotic Formula \eqref{eq:fhatkMAR312021new1} follows.
\end{proof}

Theorem \ref{thm:0} enables us to measure the decay rate of the error in approximating an  analytic, $T$-periodic function by a truncated Fourier series.
\begin{thm}[Fourier truncation error for analytic, $T$-periodic functions]\label{thm:1}
Suppose that $f \in {{\C{A}_{T,\beta}}} \foralls \beta > 0$, is approximated by the $N/2$-degree, $T$-periodic truncated Fourier series \eqref{eq:FTS1}, then 
\begin{subequations}
\begin{equation}\label{eq:Thm1}
{\left\| {f - {\Pi _N}f} \right\|} = O\left(e^{ -{\omega _{N \beta/2}}}\right),\quad \text{ as }N \to \infty.
\end{equation}
Moreover, if $f$ is $\beta$-analytic, 
then 
\begin{equation}\label{eq:Thm2}
{\left\| {f - {\Pi _N}f} \right\|} = 0,\quad \forall N \in \MBZeP,
\end{equation}
i.e., $f \in {\Span}\left\{ {{\phi _{ - N/2:N/2}}} \right\}$.
\end{subequations}
\end{thm}
\begin{proof}
Observe first that
\begin{align}
&{\left\| {f - {\Pi _N}f} \right\|^2} = \C{I}_T^{(x)}\left( {\sum\limits_{\left| k \right| > N/2} {{{\hat f}_k}{\phi _k}} \sum\limits_{\left| l \right| > N/2} {\hat f_l^*{\phi _{ - l}}} } \right)\nonumber\\
&= \sum\limits_{\left| k \right| > N/2} {\sum\limits_{\left| l \right| > N/2} {{{\hat f}_k}\hat f_l^*} } \C{I}_T^{(x)}{{\phi _{k - l}}} = \sum\limits_{\left| k \right| > N/2} {\sum\limits_{\left| l \right| > N/2} {{{\hat f}_k}\hat f_l^*} ({\phi _k},{\phi _l})}\nonumber\\
&= T\sum\limits_{\left| k \right| > N/2} {{{\left| {{{\hat f}_k}} \right|}^2}}.\label{eq:thm1ineqfor01Apr20212}
\end{align}
The above equation together with Ineq. \eqref{eq:fhatkMAR312021} yield
\begin{align}
&{\left\| {f - {\Pi _N}f} \right\|^2} \le T\left\| f \right\|_{{\C{A}_{T,\beta}  }}^2\sum\limits_{\left| k \right| > N/2} {{e^{ - 2{\omega _{\left| k \right|\beta }}}}} = \frac{{2T\left\| f \right\|_{{\C{A}_{T,\beta}  }}^2{e^{ - {\omega _{N\beta }}}}}}{{{e^{2{\omega _\beta }}} - 1}}\\
&\Rightarrow \left\| {f - {\Pi _N}f} \right\| \le \sqrt {\frac{{2T}}{{{e^{2{\omega _\beta }}} - 1}}} {\left\| f \right\|_{{\C{A}_{T,\beta}  }}}{e^{ - {\omega _{N\beta/2}}}},\label{eq:thm1ineqfor01Apr20211}
\end{align}
from which Formulas \eqref{eq:Thm1} and \eqref{eq:Thm2} are derived.
\end{proof}

FPSI methods often introduce aliasing errors. The following theorem shows that the aliasing error, for analytic, $T$-periodic functions, decays faster than $N^{-s} \foralla s \in \MBZ^+$; in fact, the aliasing error has the same order of convergence of Fourier truncation error.

\begin{thm}[Fourier aliasing error for analytic, $T$-periodic functions]\label{thm:2}
Suppose that $f \in {{\C{A}_{T,\beta}}} \foralls \beta > 0$, then 
\begin{subequations}
\begin{equation}\label{eq:FTS1AE1}
\left\| {{E_Nf}} \right\| = O\left( {{e^{ - {\omega _{N\beta/2}}}}} \right),\quad \text{as }N \to \infty,
\end{equation}
where ${E_Nf}(x) = \left( {{\Pi _N}f - {I_N}f} \right)(x)$ is the aliasing error in approximating $f$ by the $T$-periodic Fourier interpolant $I_Nf,\,\forall N \in \MBZeP$. Moreover, if $f$ is analytic on ${\F{C}_{T,\infty}}$, 
then 
\begin{equation}\label{eq:Thm2AE2}
\left\| {{E_Nf}} \right\| = 0,\quad \forall N \in \MBZeP.
\end{equation}
\end{subequations}
\end{thm}
\begin{proof}
Replacing $f$ in \eqref{DFP1} by its Fourier series yields
\begin{align}
&{\tilde f_k} = \frac{1}{N}\sum\limits_{j = 0}^{N - 1} {\left[ {\sum\limits_{l \in \MBZ} {{{\hat f}_l}\,{\phi _l}({x_j})} } \right]{\phi _{ - k}}({x_j})}  = \sum\limits_{l \in \MBZ} {{{\hat f}_l}\left[ {\frac{1}{N}\sum\limits_{j = 0}^{N - 1} {{\phi _{l - k}}({x_j})} } \right]}\nonumber\\
&= {\left[ {\sum\limits_{l \in \MBZ} {{{\hat f}_l}{\delta _{l - k,pN}}} } \right]_{\left| p \right| \in \MBZzer^ + }} = {\hat f_k} + \sum\limits_{p \in \canczer{\MBZ}} {{{\hat f}_{k + pN}}} ,\quad \forall k \in {\MB{K}'_N}.\label{eq:ftildekMAR312021}
\end{align}
Formulas \eqref{eq:fhatkMAR312021}, \eqref{eq:ftildekMAR312021}, and the Triangle Difference Ineq. imply that
\begin{align}
&{\left| {\left\| {{E_Nf}} \right\| - \left\| {{{\hat f}_{N/2}}{\phi _{N/2}}} \right\|} \right|^2} \le {\left\| {{E_Nf} - {{\hat f}_{N/2}}{\phi _{N/2}}} \right\|^2}\nonumber\\
&= \C{I}_T^{(x)}\left( {\sumd\sum\limits_{\left| k \right| \le N/2} {\sum\limits_{p \in \canczer{\MBZ}} {{{\hat f}_{k + pN}}{\phi _k}} }  \cdot \sumd\sum\limits_{\left| l \right| \le N/2} {\sum\limits_{p \in \canczer{\MBZ}} {\hat f_{l + pN}^*{\phi _{ - l}}} } } \right)\nonumber\\
\end{align}
\begin{align}
&= \sumd\sum\limits_{\left| k \right| \le N/2} {\sumd\sum\limits_{\left| l \right| \le N/2} {\sum\limits_{p \in \canczer{\MBZ}} {{{\hat f}_{k + pN}}\sum\limits_{p \in \canczer{\MBZ}} {\hat f_{l + pN}^*} } } } \C{I}_T^{(x)}{\phi _{k - l}} = \sumd\sum\limits_{\left| k \right| \le N/2} {{{\left| {\sum\limits_{p \in \canczer{\MBZ}} {{{\hat f}_{k + pN}}} } \right|}^2}{{\left\| {{\phi _k}} \right\|}^2}}\nonumber\\
&= T\sumd\sum\limits_{\left| k \right| \le N/2} {{{\left| {\sum\limits_{p \in \canczer{\MBZ}} {{{\hat f}_{k + pN}}} } \right|}^2}} \le T\sumd\sum\limits_{\left| k \right| \le N/2} {\sum\limits_{p \in \canczer{\MBZ}} {{{\left| {{{\hat f}_{k + pN}}} \right|}^2}} }\nonumber\\ &\le T\left\| f \right\|_{{\C{A}_{T,\beta} }}^2\sumd\sum\limits_{\left| k \right| \le N/2} {\sum\limits_{p \in \canczer{\MBZ}} {{e^{ - 2{\omega _{\left| {k + pN} \right|\beta }}}}} }\label{eq:proof1}\\
&= T\left\| f \right\|_{{\C{A}_{T,\beta} }}^2\left[ 2 {\sum\limits_{k = 0}^{N/2} {{\sum\limits_{p \in \canczer{\MBZ}} {{e^{ - 2{\omega _{\left| {k + pN} \right|\beta }}}}} }}  - \sum\limits_{p \in \canczer{\MBZ}} \left({{e^{ - 2{\omega _{\left| {pN} \right|\beta }}}}} + {{e^{ - 2{\omega _{\left| {N/2 + pN} \right|\beta }}}}}\right) } \right]\nonumber\\
&= T\left\| f \right\|_{{\C{A}_{T,\beta} }}^2\left[ 2 {\sum\limits_{k = 0}^{N/2} \left(e^{-2 \omega_{k \beta}}{{\sum\limits_{p \ge 1} {{e^{ - 2{\omega _{pN\beta}}}}}} + {e^{-2 \omega_{-k \beta}}\sum\limits_{p \ge 1} {{e^{ - 2{\omega _{pN\beta}}}}}}}\right)}\right.\nonumber\\
&\left.\text{}\qquad - 2\left(1+\cosh (\omega_{N \beta})\right)\sum\limits_{p \ge 1} {{e^{ - 2{\omega _{pN\beta }}}}} \right]\nonumber\\
&= T\left\| f \right\|_{{\C{A}_{T,\beta} }}^2\left[ {4\sum\limits_{k = 0}^{N/2} {\left( {\cosh (2{\omega _{k\beta }})\sum\limits_{p \ge 1} {{e^{ - 2{\omega _{pN\beta }}}}} } \right) - {e^{ - {\omega _{N\beta }}}}\coth ({\omega _{N\beta/2}})} } \right]\nonumber
\end{align}
\begin{align}
&= T\left\| f \right\|_{{\C{A}_{T,\beta} }}^2\left[ {\frac{4}{{{e^{2{\omega _{N\beta }}}} - 1}}\sum\limits_{k = 0}^{N/2} {\cosh (2{\omega _{k\beta }}) - {e^{ - {\omega _{N\beta }}}}\coth ({\omega _{N\beta/2}})} } \right]\nonumber\\
&= T \coth ({\omega _\beta }) \left\| f \right\|_{{\C{A}_{T,\beta} }}^2{e^{ - {\omega _{N\beta }}}}.\nonumber\\
&\Rightarrow \left\| {{E_Nf} - {{\hat f}_{N/2}}{\phi _{N/2}}} \right\| \le \sqrt {T\coth ({\omega _\beta })}\,{\left\| f \right\|_{{\C{A}_{T,\beta} }}}{e^{ - {\omega _{N\beta/2}}}}.\nonumber
\end{align}
Since $\displaystyle{\left\| {{{\hat f}_{N/2}}{\phi _{N/2}}} \right\| = \sqrt T \left| {{{\hat f}_{N/2}}} \right| = \sqrt T\,{\left\| f \right\|_{{\C{A}_{T,\beta} }}}{e^{ - {\omega _{N\beta/2}}}}}$ by Eq. \eqref{eq:fhatkMAR312021}, then 
\begin{align}
\left\| {{E_Nf}} \right\| &\le \left\| {{{\hat f}_{N/2}}{\phi _{N/2}}} \right\| + \left\| {{E_Nf} - {{\hat f}_{N/2}}{\phi _{N/2}}} \right\|\nonumber\\
&\le \sqrt T \left( {1 + \sqrt {\coth ({\omega _\beta })} } \right) {\left\| f \right\|_{{\C{A}_{T,\beta} }}}{e^{ - {\omega _{N\beta/2}}}},\label{eq:thm2ineqfor01Apr20211}
\end{align}
from which Formulas \eqref{eq:FTS1AE1} and \eqref{eq:Thm2AE2} are derived.
\end{proof}

Now we are ready to state the following two main results of this section: the Fourier interpolation and quadrature errors for analytic, $T$-periodic functions.

\begin{cor}[Fourier interpolation error for analytic, $T$-periodic functions]\label{cor:1}
Let $N \in \MBZeP$ and suppose that $f \in {{\C{A}_{T,\beta}}} \foralls \beta > 0$, then 
\begin{subequations}
\begin{equation}\label{eq:FTS1AE1hi1}
\left\| f - {{I_Nf}} \right\| = O\left( {{e^{ - {\omega _{N\beta/2}}}}} \right),\quad \text{as }N \to \infty.
\end{equation}
Moreover, if $f$ is $\beta$-analytic, 
then 
\begin{equation}\label{eq:Thm2hi2}
{\left\| {f - {I _N}f} \right\|} = 0.
\end{equation}
\end{subequations}
\end{cor}
\begin{proof}
The proof is established from the fact that the aliasing error is orthogonal to the truncation error $f - \Pi_N f$ in the $L^2$ sense:
\begin{align*}
&{\left\| {f - {I_N}f} \right\|^2} = {\left\| {f - {\Pi _N}f} \right\|^2} + {\left\| {{E_N}f} \right\|^2}
\end{align*}
\begin{align*}
&\le T\left[ {\frac{2}{{{e^{2{\omega _\beta }}} - 1}} + {{\left( {\sqrt {\coth ({\omega _\beta })}  + 1} \right)}^2}} \right]\left\| f \right\|_{{\C{A}_{T,\beta }}}^2{e^{ - {\omega _{N\beta }}}}\\
&\Rightarrow \left\| {f - {I_N}f} \right\| \le {\mu _{T,\beta }} \left\| f \right\|_{{\C{A}_{T,\beta }}} {e^{ - {\omega _{N\beta/2}}}},
\end{align*}
where $\displaystyle{{\mu _{T,\beta }} = \sqrt {2T\left( {\sqrt {\coth ({\omega _\beta })}  + \coth ({\omega _\beta })} \right)}}$. The upper bound collapses to zero if $f$ is $\beta$-analytic 
as proved by Theorems \ref{thm:1} and \ref{thm:2}.
\end{proof}

We conclude this section with some useful convergence results on the error of Fourier PS quadrature (FPSQ) constructed through the FPSI matrix.
\begin{cor}[FPSQ error for analytic, $T$-periodic functions]\label{cor:2}
Let $N \in \MBZeP$ and suppose that $f \in {{\C{A}_{T,\beta}}} \foralls \beta > 0$, then 
\begin{subequations}
\begin{equation}\label{eq:FPSQEFATPF1}
\left| {\C{I}_{{{\bm{x}_N}}}^{(x)}f - \Fthe f_{0:N - 1}} \right| = O\left( {{e^{ - {\omega _{N\beta/2}}}}\bmone_N} \right),\quad {\text{as }}N \to \infty.
\end{equation}
Moreover, if $f$ is $\beta$-analytic, 
then 
\begin{equation}\label{eq:FPSQEFATPF2}
\left| {\C{I}_{{{\bm{x}_N}}}^{(x)}f - \Fthe f_{0:N - 1}} \right| = \bmzer_N.
\end{equation}
\end{subequations}
\end{cor}
\begin{proof}
The proof is established through Corollary \ref{cor:1}, the Triangle Ineq., and Schwarz's Ineq. by realizing that
\begin{align}
\left| {\C{I}_{{{\bm{x}_N}}}^{(x)}f - \Fthe f_{0:N - 1}} \right| &= \left| {\C{I}_{{{\bm{x}_N}}}^{(x)}f - \C{I}_{{{\bm{x}_N}}}^{(x)}(I_Nf) + \C{I}_{{{\bm{x}_N}}}^{(x)}(I_Nf) - \Fthe f_{0:N - 1}} \right|\nonumber\\
&\le \left| {\C{I}_{{{\bm{x}_N}}}^{(x)}f - \C{I}_{{{\bm{x}_N}}}^{(x)}(I_Nf)}\right| + \left|{\C{I}_{{{\bm{x}_N}}}^{(x)}(I_Nf) - \Fthe f_{0:N - 1}} \right|\label{eq:pointhere1}\\[-1em]
&\le \left\| f - {{I_Nf}} \right\| \sqrt{\bm{x}_N} \le \mu_{T,\beta} \left\| f \right\|_{{\C{A}_{T,\beta }}} {e^{ - {\omega _{N\beta/2}}}}\,\sqrt{\bm{x}_N},\\[-2.5em]
\end{align}
which decay to zero if $f$ is $\beta$-analytic. 
\end{proof}
Corollary \ref{cor:2} shows that
\begin{equation}\label{eq:Uppbnd1}
\left\| {\C{I}_{{{\bm{x}_N}}}^{(x)}f - \Fthe f_{0:N - 1}} \right\|_2 \le \mu_{T,\beta} \left\| f \right\|_{{\C{A}_{T,\beta }}} {e^{ - {\omega _{N\beta/2}}}}\,\left\|\sqrt{\bm{x}_N}\right\|_2,
\end{equation}
$\forall N \in \MBZeP, f \in {{\C{A}_{T,\beta}}}: \beta > 0$, where $\left\|\cdot\right\|_2$ is the Euclidean norm for vectors. We refer to the upper bound \eqref{eq:Uppbnd1} by the \textit{``FPSQ error upper bound for analytic, $T$-periodic functions,''} or shortly by the \textit{``FPSQ-ATP error upper bound.''} Notice that the factor $\mu_{T,\beta}$ is a monotonically decreasing function for increasing values of $\beta$, since 
\[{\mu'_{T,\beta }} = - \frac{{\pi \left( {2\sqrt {\coth \left( {{\omega _\beta }} \right)}  + 1} \right) \csch^2\left( {{\omega _\beta }} \right)}}{{\sqrt 2 \sqrt {\coth \left( {{\omega _\beta }} \right)} \sqrt {T\left( {\coth \left( {{\omega _\beta }} \right) + \sqrt {\coth \left( {{\omega _\beta }} \right)} } \right)} }} < 0,\]
$\forall \beta  > 0$ with ${\lim _{\beta  \to \infty }}{\mu _{T,\beta }} = 2\sqrt T$. In fact, ${\mu _{T,\beta }} \to 2\sqrt T$ exponentially fast, since $\coth(\omega_{\beta})$ exponentially converge to $1$ as $\beta \to \infty$. Therefore, the larger the $\beta$-value, the much faster the error convergence rate, as both factors ${\mu _{T,\beta }}$ and ${e^{ - {\omega _{N\beta/2}}}}$ decay exponentially fast with $\beta$. We therefore say that Fourier quadrature approximation has \textit{``an infinite order accuracy.''}

Figure \ref{fig:FAccuracy1} shows the infinity- and Euclidean- error norms in the FPSQ approximation of the three functions ${f_1}(t) =$\\ $\displaystyle{2\sin (3t - 1) + 1 \in \C{A}_{2 \pi/3,\infty}}$, $\displaystyle{{f_2}(t) = \frac{1}{{2 - \cos t}} \in \C{A}_{2 \pi,{\left[\text{Im}\left( {{{\cos }^{ - 1}}(2)} \right)\right]^ - }}}$, and $\displaystyle{{f_3}(t) = \frac{1}{{{{\sin }^2}t + 16}} \in \C{A}_{\pi,{\left[ {{{\sinh }^{ - 1}}(4)} \right]^ - }}}$, for $N = 10:10:100$, where Im$(z)$ denotes the imaginary part of a complex number $z$ and $a^-$ means a number sufficiently close to $a$ and less than $a \foralla a \in \MBR$. Since $f_1$ is $\beta$-analytic, Corollary \ref{cor:2} indicates that the Euclidean-error of the FPSQ is zero using exact-arithmetic. However, in double-precision arithmetic, where numbers are represented in 15 significant figures, we would expect the errors to immediately plateau at an $O\left(10^{-15}\right)$ level and cease to decrease but instead bounce around randomly, as seen in the lower left plot of the figure. This plateau is the result of the accumulation of round-off error in the FPSQ computation. Observe also that the present algorithm produces generally smaller errors compared to \cite[Algorithm 3.1]{elgindy2019high} due to the slight reduction in computational cost as argued in \cite{Elgindy2022b}. On the other hand, the function $f_2$ is analytic on ${\FOmega_{2 \pi}}$ but not entire, since it has poles at $2 \pi n \pm \cos^{-1}(2)\,\forall n \in \MBZ$ in the complex plane. Also, the function $f_3$ is analytic on ${\FOmega_{\pi}}$ but not entire, since it has poles at $2 \pi n \pm i \sinh^{-1}(4)\,\forall n \in \MBZ$. Therefore, $f_3$ has a radius of analyticity that is larger than that of $f_2$ by about $60\%$. Combining this data with the fact that the FPSQ-ATP error upper bound decays exponentially fast as $\beta$ grows larger, we foresee that the quadrature errors of $f_3$ will collapse and reach the plateau level much faster than the quadrature errors of $f_2$, for increasing values of $N$. This is clearly observed in the figure where the errors rapidly decay for $f_3$ and plateaus very early near $N = 20$ compared with the quadrature errors associated with $f_2$ which cease to decrease at around $N = 50$. Notice in both cases that the FPSQ-ATP error upper bound is larger than the observed quadrature errors until the plateau is reached, in agreement with the theoretical results of Corollary \ref{cor:2}.

\begin{figure*}
\centering
\includegraphics[scale=0.35]{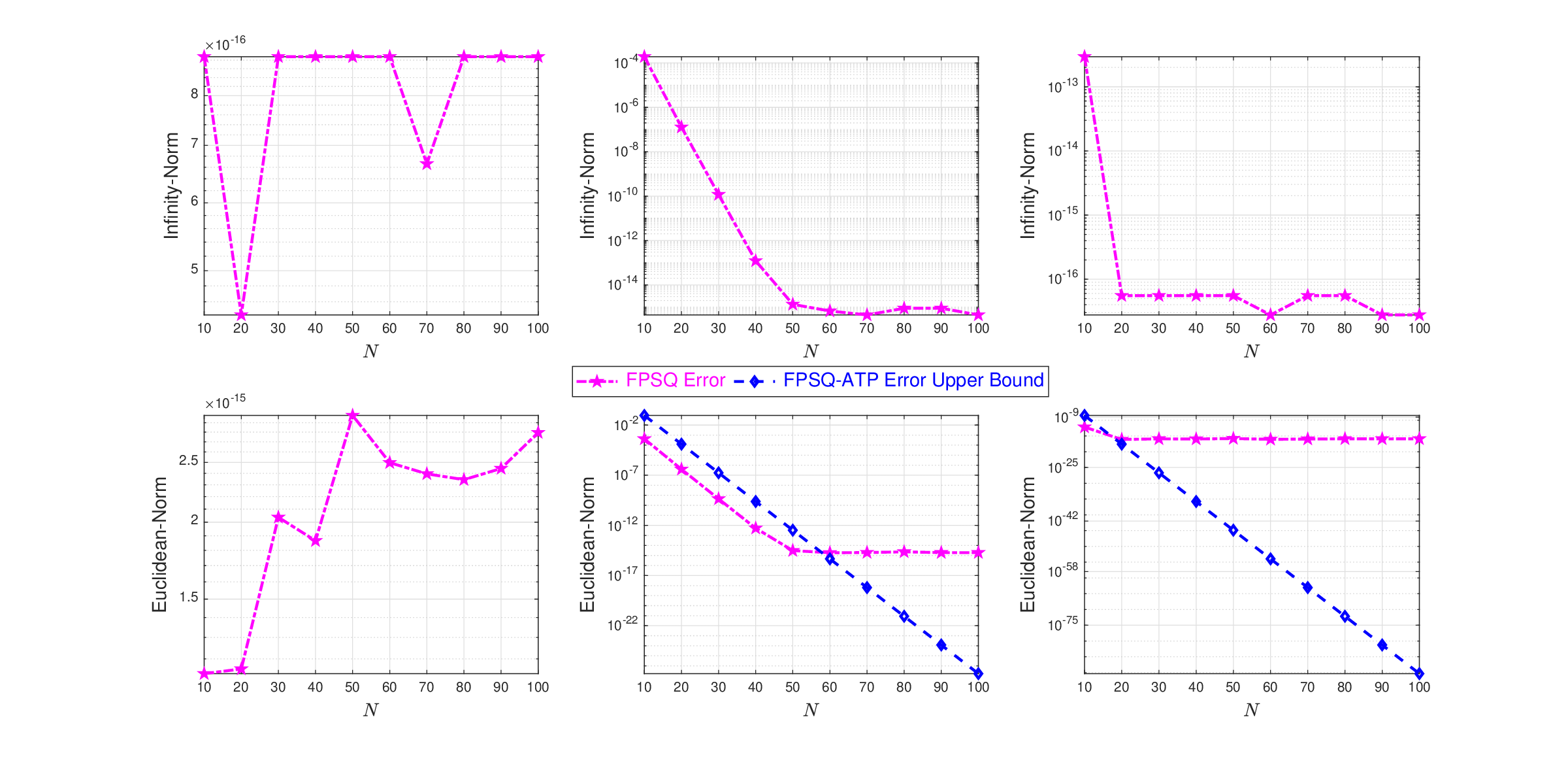}
\caption{The figure shows the infinity- and Euclidean- error norms in the FPSQ approximation of the three functions $\displaystyle{{f_1}(t) = 2\sin (3t - 1) + 1}$ (left upper and lower plots), $\displaystyle{{f_2}(t) = \frac{1}{{2 - \cos t}}}$ (middle upper and lower plots), and $\displaystyle{{f_3}(t) = \frac{1}{{{{\sin }^2}t + 16}}}$ (right upper and lower plots), for $N = 10:10:100$.}
\label{fig:FAccuracy1}
\end{figure*}

\section{The FIPS Method}
\label{sec:FPIMIRF1}
In this section we demonstrate the FIPS approach to discretize Problem $\C{IP}$. Using Formula \eqref{eq:PCase13Feb221} and following the writing convention introduced in Section \eqref{sec:PN}, we can approximate the performance index by the following formula:
\begin{equation}\label{eq:Dperind1}
J \approx J_N = \frac{1}{T} \Fthe_N\,g(\bmx(\bmt_N), \bmu(\bmt_N), \bmt_N) = \frac{1}{N} \bmone_N^t\,g(\bmx(\bmt_N), \bmu(\bmt_N), \bmt_N).
\end{equation}
Next, let $\bmF = [\bmF_{1, N}, \bmF_{2, N}, \ldots, \bmF_{n, N}]: \bmF_{i, N} = f_{i, 0:N-1}$ and $f_{i,j} = f_i(\bmx(t_{j}), \bmu(t_{j}), t_{j})\,\forall i \in \MBJ'_n, j \in \MBJ_N$, where $f_i$ is the $i$th-element of $\bmf\,\forall i$. Then
\begin{equation}
\C{I}_{t_i}^{(\tau)} {\bmf\cancbra{\bmx(\tau), \bmu(\tau), \tau}} \approx [\Fthe_i \bmF_{1,N}, \Fthe_i \bmF_{2,N}, \ldots, \Fthe_i \bmF_{n,N}]^t = (\Fthe_i \bmF)^t,
\end{equation}
$\forall i \in \MBJ_N$. Therefore, 
\begin{equation}\label{eq:ItN1}
\C{I}_{\bmt_N}^{(\tau)} {\bmf\cancbra{\bmx(\tau), \bmu(\tau), \tau}} \approx [(\Fthe_0 \bmF)^t, (\Fthe_1 \bmF)^t, \ldots, (\Fthe_{N-1} \bmF)^t]^t = \Fthe \bmF.
\end{equation}
Since 
\begin{align}
\bmF &= [\bmF_{1, N}, \bmF_{2, N}, \ldots, \bmF_{n, N}] = [\bmF_{1, N}^t; \bmF_{2, N}^t; \ldots; \bmF_{n, N}^t]^t\nonumber\\
&= \bmf(\bmx(\bmt_N), \bmu(\bmt_N), \bmt_N),\label{eq:IntRem1}
\end{align}
then Eqs. \eqref{eq:ItN1} and \eqref{eq:IntRem1} imply
\begin{equation}
\C{I}_{\bmt_N}^{(\tau)} {\bmf\cancbra{\bmx(\tau), \bmu(\tau), \tau}} \approx \Fthe \bmf(\bmx(\bmt_N), \bmu(\bmt_N), \bmt_N).
\end{equation}
Hence, collocating the system dynamics in integral form \eqref{eq:AppRedFIM1} at the Fourier nodes set $\MBS_N$ yields the following system of algebraic equations:
\begin{equation}\label{eq:DiscDynSysK1}
\bmx(\bmt_N) \approx \bmx(0) \otimes \bmone_N + \Mvec\left(\Fthe\,\bmf(\bmx(\bmt_N), \bmu(\bmt_N), \bmt_N)\right).
\end{equation}
Finally, the inequality path constraints at $\MBS_N$ reads
\begin{equation}\label{eq:hhineqcons1}
\bmc(\bmx(\bmt_N), \bmu(\bmt_N), \bmt_N) \le \F{O}_{N, p}.
\end{equation}
Problem $\C{IP}$ is now converted into a constrained NLP in which the goal is to minimize the discrete performance index \eqref{eq:Dperind1} subject to the equality constraints \eqref{eq:DiscDynSysK1} and the inequality constraints \eqref{eq:hhineqcons1}. The constrained NLP can be solved using well-developed optimization software for the unknowns $\bmx(\bmt_N)$ and $\bmu(\bmt_N)$, and the approximate optimal state and control variables can then be calculated at any point $t \in \FOmega_T$ through the Fourier PS expansions
\begin{equation}\label{eq:17}
\bmx(t) = \bmx\left(\bmt_N^t\right) \bsC{F}_{N}(t)\quad \text{and}\quad\bmu(t) = \bmu\left(\bmt_N^t\right) \bsC{F}_{N}(t),
\end{equation}
where $\bsC{F}_{N}(t) = \C{F}_{0:N-1}\cancbra{t}$. In this work, we shall use MATLAB fmincon solver employing the ``interior-point'' algorithm, considered to be one of the most efficient algorithms in numerical optimization, to solve the constrained NLP. We refer to the FIPS method applied together with the above NLP solver by the FIPS-fmincon method; moreover, the approximate optimal state- and control-variables obtained by this method are denoted by by $\tbmxs$ and $\tbmus$, respectively. Notice here that all integrals involved in Problem $\C{IP}$ were efficiently computed using matrix vector multiplication without the need for the FFT, since the FIM can provide nearly exact integrals with exponential rates of convergence using relatively coarse mesh grids as proven in Section \ref{sec:CRFSPF1}. This also shows that the FIPS method converges exponentially fast to the smooth, $T$-periodic solutions of Problem $\C{P}$, since the quadrature errors induced by the FIM are the main source of errors in the NLP described by \eqref{eq:Dperind1}, \eqref{eq:DiscDynSysK1}, and \eqref{eq:hhineqcons1}, which decay with exponential rates by Corollary \eqref{cor:2} and remarkably vanish for $\beta$-analytic integrand functions.

\section{Computational Results}
\label{sec:CRAC1}
In this section we apply the developed FIPS-fmincon method on two test problems well studied in the literature. The first problem has no inequality constraints on the state and control variables, while the second problem has lower bounds on the states and controls. Both test problems are instances of Problem $\C{P}$. All computations were carried out using MATLAB R2022a software installed on a personal laptop equipped with a 2.9 GHz AMD Ryzen 7 4800H CPU and 16 GB memory running on a 64-bit Windows 11 operating system. The fmincon solver was performed with initial guesses of all ones and was stopped whenever
\[\left\| {{\bm{x}^{(k + 1)}} - {\bm{x}^{(k)}}} \right\|_2 < {10^{ - 15}}\quad \text{or}\quad \left\|J_N^{(k+1)} - J_N^{(k)} \right\|_2 < {10^{ - 15}},\]
where $\bm{x}^{(k)}$ and $J_N^{(k)}$ denote the approximate minimizer and optimal cost function value at the $k$th iteration, respectively. The quality of the approximations are measured using the absolute discrete feasibility error (ADFE) \footnote{The ADFE at a collocation node measures the amount by which the approximate optimal solutions fail to satisfy the discrete constraint equations \eqref{eq:DiscDynSysK1} being used for the approximation at that node.}, $\bm{\C{E}}_N = (\C{E}_i)_{0 \le i \le n N-1}$, at the collocation points, which we define by
\begin{equation}\label{eq:DiscDynSysK2}
\bm{\C{E}}_N = \left|\bmx(0) \otimes \bmone_N + \Mvec\left(\Fthe\,\bmf(\bmx(\bmt_N), \bmu(\bmt_N), \bmt_N)\right) - \bmx(\bmt_N)\right|.
\end{equation}

\textbf{Problem 1.} Consider Problem $\C{P}$ with ${g}(\bmx(t), \bmu(t), t) = 0.5 x_1^2$ $+0.25 x_2^4-0.5 x_2^2+0.5 b u^2$ and $\bmf\left(\bmx(t),\bmu(t),t\right) = [x_2; u]$ with free periodic boundary conditions and a weighting parameter $b$ for the control variable. For $b \le 0.25$, the $\pi$-test guarantees the existence of periodic solutions which improve the performance index more than the static extremal solutions. Moreover, the quartic term in the integral cost function dominates the negative quadratic term for large values of $x_2$, thus bounding the amplitude of the optimal minimizing solutions; cf. \cite{evans1987solution}. This problem was solved earlier by \citet{evans1987solution} using the Lindstedt-Poincar\'{e} asymptotic series expansion and later by \citet{Elnagar2004707} using a FPS method. Figure \ref{fig:Fig1} shows the profiles of the approximate optimal state and control variables and the ADFEs obtained at the collocated nodes set $\MBS_{12}$ using the FIPS-fmincon method, for $T = 4.431736$ and $b = 0.2475$. The elapsed time required to run the FIPS-fmincon method and obtain the required approximations is about 0.209 seconds (s). Table \ref{tab:OFICAAPOD1} shows the approximate minimum values of the performance index obtained by the current method and the method of \cite{Elnagar2004707} for several values of $b, T$, and $2 m = N$. Clearly, the approximate optimal objective function values obtained by the FIPS-fmincon method are lower than those obtained by the method of \cite{Elnagar2004707} for all parameter variations shown in the table with negligible ADFEs that are approaching the machine epsilon.

\begin{figure}
\centering
\includegraphics[scale=0.24]{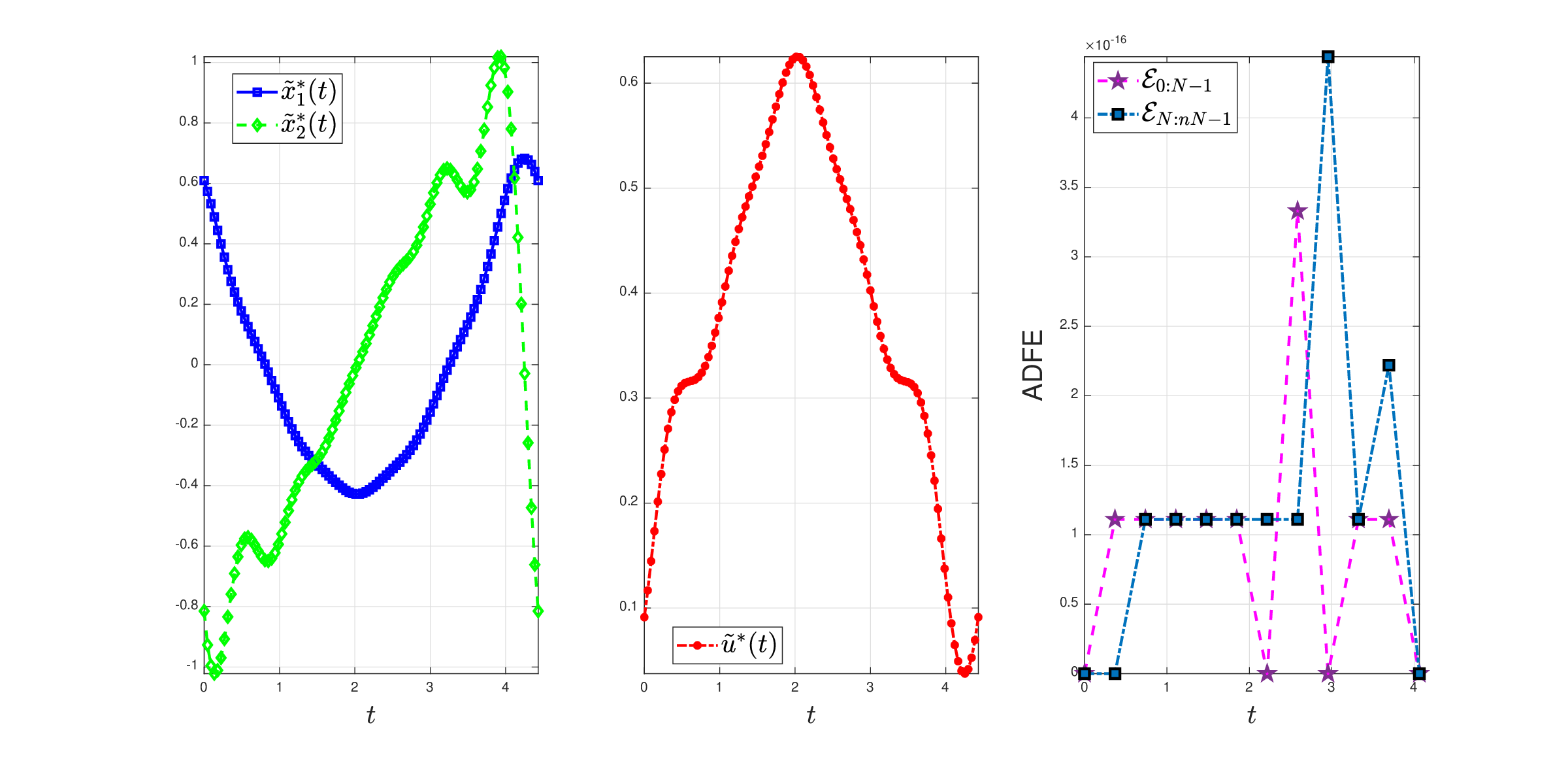}
\caption{The profiles of the approximate optimal state- and control-variables and the ADFEs obtained at the collocated nodes set $\MBS_{12}$ for Problem 1 using the FIPS-fmincon method. Here $n = 2, N = 12$, and the plots of the state and control variables were generated using 100 linearly spaced nodes in $\FOmega_{4.431736}$.}
\label{fig:Fig1}
\end{figure}

\begin{table}[t]
\caption{Comparisons between the FIPS-fmincon method and the method of \cite{Elnagar2004707} in terms of the approximate performance index minimum values obtained for several values of $b, T$, and $2 m = N$. The values of $J_{2m}, J_N$, and $\left\|ADFE\right\|_{\infty}$ are rounded to the displayed number of significant digits.} 
\centering 
\resizebox{0.9\columnwidth}{!}{%
\begin{tabular}{*{6}{c}} 
\hline\hline 
\multicolumn{6}{c}{Problem 1}\\
 & & & Method of \cite{Elnagar2004707} & \multicolumn{2}{c}{The FIPS-fmincon Method}\\
\cmidrule{5-6}
$b$ & $T$ & $2 m = N$ & $J_{2 m}$ & $J_N$ & $\left\|\text{ADFE}\right\|_{\infty}$\\[0.5ex]
\hline\hline
0.2475 & 4.431736 & 12 & $-4.18880200 \times 10^{-6}$ & $-4.02232772 \times 10^{-2}$ & $4.441 \times 10^{-16}$\\
0.2475 & 4.43173625 & 16 & $-4.18880203 \times 10^{-6}$ & $-4.08422489 \times 10^{-2}$ & $4.441 \times 10^{-16}$\\
0.2250 & 4.32786300 & 12 & $-4.38802000 \times 10^{-4}$ & $-4.42143743 \times 10^{-2}$ & $2.220 \times 10^{-16}$\\
0.2250 & 4.32786260 & 16 & $-4.38880203 \times 10^{-4}$ & $-4.48293692 \times 10^{-2}$ & $6.661 \times 10^{-16}$\\
0.1000 & 3.6343100 & 12 & $-1.97810000 \times 10^{-2}$ & $-7.75663473 \times 10^{-2}$ & $2.220 \times 10^{-16}$\\
0.1000 & 3.6343132 & 16 & $-1.97813000 \times 10^{-2}$ & $-7.81094298 \times 10^{-2}$ & $3.331 \times 10^{-16}$\\[1ex]
\hline
\end{tabular}
}
\label{tab:OFICAAPOD1} 
\end{table}

\textbf{Problem 2 (Solar Energy Control).} This problem is taken from Refs. \cite{dorato1979periodic,epperlein2020frequency} with slight modifications: Consider the following linear control mathematical model of a collector/storage/ enclosure type solar heating system
\begin{align}
{\left( {m {C_p}} \right)_E}{{\dot T}_E} &= {Q_{aux}} + {Q_S} - {\left( {U A} \right)_E}\left( {{T_E} - {T_A}} \right),\\
{\left( {m {C_p}} \right)_S}{{\dot T}_S} &= {Q_C} - {Q_S} - {\left( {U A} \right)_S}\left( {{T_S} - {T_A}} \right),
\end{align}
where $m$ is the mass (kg), $C_p$ is the specific heat coefficient (kJ C$^{-1}$ kg$^{-1}$), $A$ is the area (m$^2$), $U$ is the heat transfer coefficient (kJ h$^{-1}$ C$^{-1}$ m$^{-2}$), $Q_S(t)$ is the heat flow rate (kJ h$^{-1}$) from storage to enclosure, $Q_C(t)$ is the insolation or the heat flow rate (kJ hr$^{-1}$) from collector to storage, $Q_{aux}(t)$ is the heat flow rate (kJ h$^{-1}$) from auxiliary heat source to enclosure, $T_A(t)$ is the ambient temperature (${}^{\circ}$C), $T_S(t)$ is the storage temperature (${}^{\circ}$C), and $T_E(t)$ is the enclosure temperature (${}^{\circ}$C). The subscripts in the parameters denote values for storage element or enclosure element. Given the 24 h-periodic disturbance inputs $Q_C(t)$ and $T_A(t)$, we seek the 24 h-periodic control inputs, $Q_{aux}(t)$ and $Q_S(t)$, which minimize the following performance index
\begin{align}
J &= \frac{1}{{24}} \C{I}_{24}^{(t)} \left[ 1000{{\left( {{T_E} - {{\bT}_E}} \right)}^2} + 10{{\left( {{T_S} - {{\bT}_S}} \right)}^2} + 0.1{{\left( {{Q_{aux}} - {{\hat Q}_{aux}}} \right)}^2}\right.\nonumber\\
&\left.\qquad + {Q_{aux}} \right],
\end{align}
which gives the average steady-state value of an integral quadratic measure; in particular, it penalizes the enclosure and storage temperatures deviations from given set points and the auxiliary heat input from its mean value. The last term, integrated over a period, represents the average 
auxiliary energy consumed. The bars and hats above the variables denote fixed set levels and average levels, respectively. A depiction of the solar heating system is shown in Figure \ref{fig:Fig2}. The parameter values used for the solar heating system are shown in Table \ref{tab:Table2}. The ambient temperature and insolation are given by $T_A(t) = -10 \sin(\omega t)$ and $Q_C(t) = 13333 [1-\cos(\omega t)]$, respectively, where $\omega = \pi/12$ h$^{-1}$. Some care is required in this problem in order to avoid negative values of $Q_{aux}$ and $Q_S$ and unreasonable values of temperatures. We assume here that $Q_{aux}(t) \ge 8,000$ kJ h$^{-1}, Q_S(t) > 0$ kJ h$^{-1}$, and  $T_E(t), T_S(t) \ge 0 {}^{\circ}$C$\,\forall t \in \FOmega_{24}$.

\begin{figure}[t]
\centering
\includegraphics[scale=0.35]{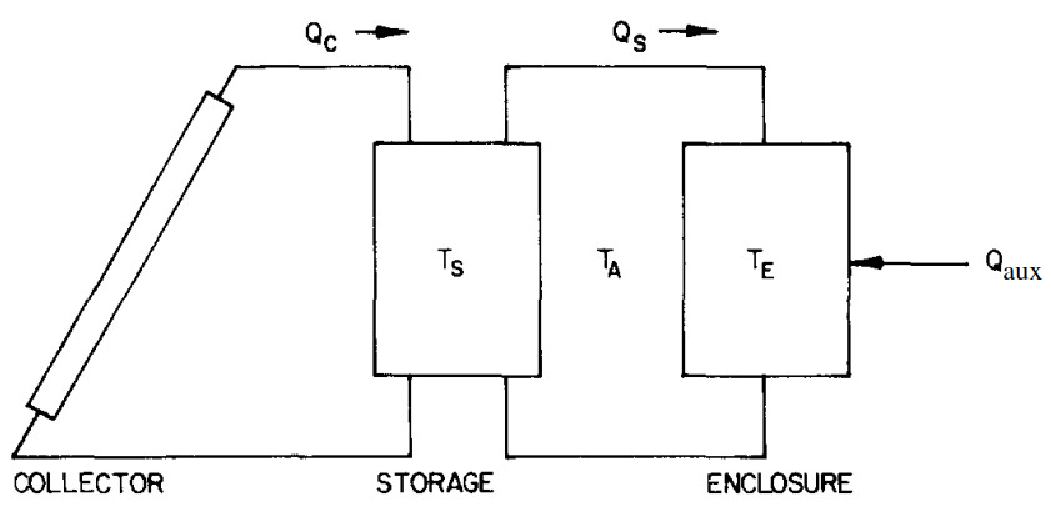}
\caption{A graph of the solar heating system as taken from Ref. \cite{dorato1979periodic}. The capital initials ``AUX'' are replaced with small letters.}
\label{fig:Fig2}
\end{figure}

\begin{table}[t]
\caption{Parameter values used for the solar heating system.} 
\centering 
\resizebox{0.5\columnwidth}{!}{%
\begin{tabular}{*{4}{c}} 
\hline\hline 
Parameter & $S$ & $E$ & Unit\\[0.5ex]
\hline\hline
$U A$ & 20.07 & 949.5 & kJ/(${}^{\circ}$C h)\\
$m C_p$ & 19000 & 18890 & kJ/(${}^{\circ}$C)\\
$\bT$ & 30 & 20 & ${}^{\circ}$C\\[1ex]
\hline
\end{tabular}
}
\label{tab:Table2} 
\end{table}

This problem has the form laid out in Section \ref{sec:PS1} with $\bmx = [T_E; T_S], \bmu = [Q_{aux}; Q_S], g\left(\bmx(t),\bmu(t),t\right) = 1000 (x_1-20)^2+10 (x_2-30)^2$ $+0.1(u_1-\hu_1)^2+u_1$,
\begin{equation}
{\bm{f}}\left(\bmx(t),\bmu(t),t\right) = \left[ {\begin{array}{*{20}{c}}
\displaystyle{{\frac{1}{{{{\left( {m{C_p}} \right)}_E}}}}\left( {{u_1} + {u_2} - {{\left( {UA} \right)}_E}\left( {{x_1} - {T_A}} \right)} \right)}\\
\displaystyle{{\frac{1}{{{{\left( {m{C_p}} \right)}_S}}}}\left( {{Q_C} - {u_2} - {{\left( {UA} \right)}_S}\left( {{x_2} - {T_A}} \right)} \right)}
\end{array}} \right],
\end{equation}
and $\bmc(\bmx(t), \bmu(t), t) = [-\bmx; 8000-u_1; \epsilon-u_2]$, where $\hu_1 = \frac{1}{24} \C{I}_{24}^{(t)} u_1$ is the average value of the heat flow rate from auxiliary heat source to enclosure over one period, and $\epsilon$ is a relatively small positive number to keep $u_2$ always positive. In practice, we may assume this constant to be the machine precision that is approximately equals $2.2204 \times 10^{-16}$ in double precision arithmetic. Figure \ref{fig:Fig3} shows the profiles of the approximate optimal temperatures distributions and heat flow rates as well as the absolute differences between the average temperatures and their set points and the ADFEs obtained at the collocated nodes set $\MBS_{50}$ using the FIPS-fmincon method. Observe how the FIPS-fmincon method brings $\hbmx - \bbmx$ as close to zero as possible, while on the other hand minimizes the cost of setting $\bmu$ to nonzero values by selecting $u_1$ to be nearly its lower bound. It is realistically consistent to observe through the plots that the storage is slowly heating up during the day and cools down during the night. A similar observation holds true for the enclosure which receives heat during the day due to the large thermal inertia $(m C_p)_E$.

\begin{figure*}
\centering
\includegraphics[scale=0.3]{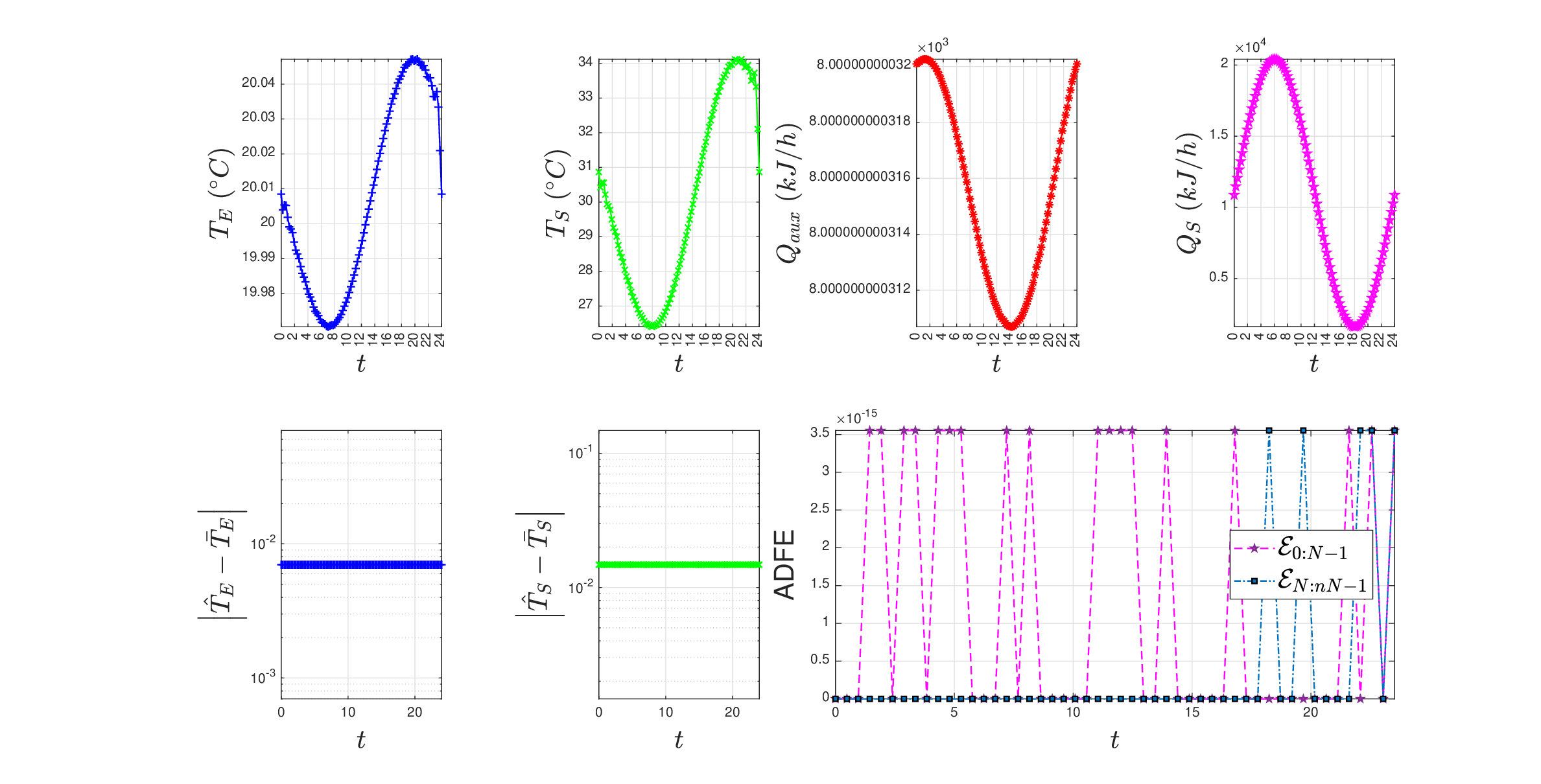}
\caption{The profiles of the approximate optimal temperatures distributions and heat flow rates (first row) as well as the absolute differences between the average temperatures and their set points and the ADFEs obtained at the collocated nodes set $\MBS_{50}$ (second row) for Problem 2 using the FIPS-fmincon method. Here $n = 2, N = 50$, and the plots of the temperatures and heat flow rates were generated using 100 linearly spaced nodes in $\FOmega_{24}$.}
\label{fig:Fig3}
\end{figure*}

\begin{rem}
For more complex OC problems with bang-bang periodic solutions, the reader may consult our recent works in \cite{Elgindy2022b,elgindy2023optimal}. For further challenging fractional OC problems with periodic solutions, we refer the reader to our recent works in \cite{elgindy2023fouriera,elgindy2023fourierb}.  
\end{rem}

\section{Conclusion}
\label{sec:Conc}
This paper presented a robust and computationally efficient Fourier-based algorithm for solving generally nonlinear, periodic OC problems. The proposed FIPS method promises further to be a very appropriate tool for the design of solar systems as demonstrated through numerical simulations. The developed method enjoys rapid convergence with exponential rate and superior flexibility in treating problems with and without inequality constraints in the absence of explicit use of extra necessary conditions of optimality. A key reason underlying the computationally streamlined nature of the approach is the automatic satisfaction of the state and control periodic conditions due to the periodicity of their Fourier interpolants. Another advantage of the approach lies in the integral reformulation of the system dynamics, which allows for using numerical integration operators widely known for their well-conditioning. The convergence of Fourier interpolation and quadrature is addressed in detail for smooth, $T$-Periodic Functions leading to the derivation of a sharp FPSQ error estimate as indicated by Ineq. \eqref{eq:Uppbnd1} and sustained by Figure \ref{fig:FAccuracy1}. The formulation and analysis presented here can be smoothly extended to other generally nonlinear, periodic OC problems in Mayer and Bolza forms. 

It is important to remind the reader that no single direct optimization method is perfect for all types of OC problems because different problems have different characteristics that make them more or less amenable to certain methods. For example, methods designed to solve OC problems with nonsmooth methods are less efficient and more computationally expensive than those designed to solve OC problems with sufficiently differentiable solutions. The latter are also less accurate than those designed to solve classes of problems with periodic, sufficiently differentiable solutions, and so on. The proposed FIPS method works best for periodic problems with sufficiently differentiable solutions, where it can can capture the smooth and periodic features of the solution very accurately with remarkable exponential convergence, even with a relatively small number of grid points. This is in contrast to finite difference and other traditional methods implementing nonperiodic basis functions like Chebyshev or Legendre bases function, for example, which typically converge more slowly than Fourier basis functions with higher costs. 

As a result of these advantages and the remarkable advantages of FIPS methods stated earlier in the Introduction Section, we consider the proposed FIPS method to be at the front line of the methods of choice for solving periodic problems with sufficiently differentiable solutions. For periodic problems with bang-bang solutions, the reader may consult \cite{Elgindy2022b,elgindy2023optimal} in which the solutions can be recovered in short time periods within excellent accuracies.

\section{Future Work}
\label{sec:FW1}
One possible future work is to explore the application of the costate mapping theory on Problem $\C{P}$ to ensure that the results of the reduced NLP obtained by the direct FIPS method meet the first-order necessary conditions of the OC problem in discrete manner. 

%
%
%
%
%
%

\section*{Declarations}
\subsection*{Competing Interests}
The author declares there is no conflict of interests.

\subsection*{Availability of Supporting Data}
The author declares that the data supporting the findings of this study are available within the article.

\subsection*{Ethical Approval and Consent to Participate and Publish}
Not Applicable.

\subsection*{Human and Animal Ethics}
Not Applicable.

\subsection*{Consent for Publication}
Not Applicable.

\subsection*{Funding}
The author received no financial support for the research, authorship, and/or publication of this article.

\subsection*{Authors' Contributions}
The author confirms sole responsibility for the following: study conception and design, data collection, analysis and interpretation of results, and manuscript preparation.

\subsection*{Acknowledgment}
Special thanks to Dr. Jonathan Epperlein\footnote{Dr. Jonathan Epperlein is a research scientist at IBM Research Europe.} for his fruitful discussions on the Solar Energy Control Problem.

\bibliographystyle{model1-num-names}
\bibliography{Bib}
\end{document}